\documentclass[12pt, reqno]{amsart}
\author[P.~Leonetti]{Paolo Leonetti}
\address{
Universit\`a degli Studi dell'Insubria, via Monte Generoso 71, Varese 21100, Italy}
\email{leonetti.paolo@gmail.com}
\makeatletter
\@namedef{subjclassname@2020}{\textup{2020} Mathematics Subject Classification}
\makeatother



\title{Can a small Gaussian perturbation break subadditivity?}

\usepackage[T1]{fontenc}
\usepackage{amsmath}
\usepackage{amssymb}
\usepackage{amsthm}
\usepackage[left=2.5cm, right=2.5cm, top=3cm]{geometry}
\usepackage{hyperref}
\usepackage{fancyhdr}
\usepackage[inline]{enumitem}
\usepackage{comment}
\usepackage{nicefrac}
\usepackage{bm}
\usepackage{mathrsfs}
\usepackage{graphicx}
\usepackage[utf8]{inputenc}
\usepackage{cancel}
\usepackage{mathtools}

\newcommand{\vertiii}[1]{{\left\vert\kern-0.25ex\left\vert\kern-0.25ex\left\vert #1 
    \right\vert\kern-0.25ex\right\vert\kern-0.25ex\right\vert}}
		
\AtBeginDocument{%
   \def\MR#1{}
}

\newtheorem{thm}{Theorem}[section]
\newtheorem{lem}[thm]{Lemma}
\newtheorem{prop}[thm]{Proposition}

\theoremstyle{definition} 
\newtheorem{defi}[thm]{Definition}%[section]
\let\olddefi\defi
\renewcommand{\defi}{\olddefi\normalfont}
\newtheorem{example}[thm]{Example}
\let\oldexample\example
\renewcommand{\example}{\oldexample\normalfont}
\newtheorem{question}[thm]{Question}
\let\oldquestion\question
\renewcommand{\question}{\oldquestion\normalfont}
\newtheorem{rmk}[thm]{Remark}
\let\oldrmk\rmk
\renewcommand{\rmk}{\oldrmk\normalfont}

\theoremstyle{remark}
\newtheorem{claim}{\textsc{Claim}}
\newtheorem*{claim*}{\textsc{Claim}}

\hypersetup{
    pdftitle={Subadditivity},
    pdfauthor={Paolo Leonetti},
    pdfmenubar=false,
    pdffitwindow=true,
    pdfstartview=FitH,
    colorlinks=true,
    linkcolor=blue,
    citecolor=green,
    urlcolor=cyan
}

\providecommand{\MR}[1]{}

\providecommand{\MR}{\relax\ifhmode\unskip\space\fi MR }

\begin{document}

\maketitle
\thispagestyle{empty}

\begin{abstract}
\noindent Given an integer $a\ge 1$, a function $f: \mathbb{R}\to \mathbb{R}$ is said to be $a$-subadditive if 
$$
f(ax+y) \le af(x)+f(y) \,\,\,\text{ for all }x,y \in \mathbb{R}.
$$
Of course, $1$-subadditive functions (which correspond to ordinary subadditive functions) are $2$-subadditive. % and $3$-subadditive. 
Answering a question of Matkowski, we show that there exists a continuous function $f$ satisfying $f(0)=0$ which is $2$-subadditive but not $1$-subadditive. 
In addition, the same example is not $3$-subadditive, which shows that the sequence of families of continuous $a$-subadditive functions passing through the origin is not increasing with respect to $a$. The construction relies on a perturbation of a given subadditive function with an even Gaussian ring, which will destroy the original subadditivity while keeping the weaker property. 

Lastly, given a positive rational cone $H\subseteq (0,\infty)$ which is not finitely generated, we prove that there exists a subadditive bijection $f:H\to H$ such that $\liminf_{x\to 0}f(x)=0$ and $\limsup_{x\to 0}f(x)=1$. This is related an open question of Matkowski and {\'S}wi{\k a}tkowski in [Proc.~Amer.~Math.~Soc.~\textbf{119} (1993), 187--197].
%Lastly, we prove that, if $H\subseteq (0,\infty)$ is a positive rational cone which is not finitely generated, then there exists a subadditive bijection $f:(0,\infty)\to (0,\infty)$ such that $\liminf_{x\to 0}f(x)=0$ and $\limsup_{x\to 0}f(x)=1$. This answers another question of Matkowski and {\'S}wi{\k a}tkowski in [Proc.~Amer.~Math.~Soc.~\textbf{119} (1993), 187--197].
\end{abstract}

\section{Introduction and Main Results}\label{sec:intr}

%Intro: \cite{MR1113646} linear, \cite{MR1009994, MR1316490, MR3061390} charact Lp spaces, \cite{MR4811070, MR2431038, MR137800} and \cite[Chapter 7]{MR423094}, other \cite{MR680171}

%Ref first question: \cite{MR1158730, MR3083197, MR1316490}. 

%\textcolor{red}{Ref second question: \cite{MR1088646, MR1234992} Important \cite{MR1176072}}

%\bigskip

Subadditive functions play an important role in many branches of mathematics, including applications in the theory of convex sets, uniqueness of differential equations, and theory of semigroups, see e.g. \cite{MR2431038, MR137800, MR1316490, MR1234992, MR680171, MR4811070}, cf. \cite[Chapter 7]{MR423094} and references therein. 

In this work, we study the existence of certain real-valued functions which are [slightly not-]subadditive. To be more precise, given an integer $a \ge 1$ (or, more generally, a real $a>0$), a function $f: \mathbb{R}\to \mathbb{R}$ is said to be $a$\emph{-subadditive} if 
\begin{equation}\label{eq:asubadditive}
\forall x,y \in \mathbb{R}, \quad \quad 
f(ax+y)\le af(x)+f(y). 
\end{equation}
In particular, $1$-subadditive functions are the ordinary subadditive functions. The families of $a$-subadditive functions 
have been studied in \cite{MR1158730, MR3083197, MR1316490}.
It turns out that properties of $a$-subadditive functions and related notions have been useful in results such as characterizations of $L_p$-norm-like functions and commutativity of certain equivalents, see e.g. \cite{MR3720973, MR1009994, MR1113646, MR3061390}. % and references therein. 

\begin{defi}\label{defi:Sa}
    For each $a>0$, let $\mathcal{S}_a$ be the family of $a$-subadditive continuous functions $f: \mathbb{R}\to \mathbb{R}$ such that $f(0)=0$. 
\end{defi}

It is immediate to see that $1$-subadditive functions are $2$-subadditive.  
%[and $a$-subadditive for every integer $a\ge 2$], 
This implies that $$\mathcal{S}_1\subseteq \mathcal{S}_2$$ and, with the same reasoning, $\mathcal{S}_1\subseteq \mathcal{S}_a$ for every integer $a\ge 2$. 
More generally, we have the following straightforward result (its proof is omitted). 
\begin{lem}\label{lem:basic}
    Fix $a_1,\ldots,a_k>0$ and suppose that $f: \mathbb{R}\to \mathbb{R}$ is $a_i$-subadditive for each $i=1,\ldots,k$. Then $f$ is $a$-subadditive for all $a$ in the additive semigroup generated by $\{a_1,\ldots,a_k\}$. 
\end{lem}
%f(5x+y)=f(3x+2x+y) \le 3f(x)+f(2x+y) \le 5f(x)+f(y)
As shown in \cite[Example 2]{MR3083197}, the function $f$ defined by 
$$
\forall x \in \mathbb{R}, \quad 
f(x):=1+2\cdot \bm{1}_{\mathbb{Q}}(x)
$$
%$f(x):=1+2\cdot \bm{1}_{\mathbb{Q}}(x)$ for each $x \in \mathbb{R}$ 
is $2$-subadditive and not $1$-subadditive. 
%Note that $f$ is discontinuous everywhere. 
Motivated by the nice regularity properties of $a$-subadditive functions, Janusz Matkowski asked 
in \cite[p. 53]{MR3083197} 
%during the open problem sessions of the 49th International Symposium on Functional Equations (Austria, 2011) and of the 57th International Symposium on Functional Equations (Poland, 2019) 
whether an analogue example exists under the mild regularity conditions given in Definition \ref{defi:Sa}. %It is worth remarking that the same question appears also in \cite[p. 53]{MR3083197}: 
It is worth remarking that the same question has been posed during the open problem sessions of the 49th International Symposium on Functional Equations (Austria, 2011) and of the 57th International Symposium on Functional Equations (Poland, 2019): 
\begin{question}\label{q:matk1}
    Is it true that $\mathcal{S}_1=\mathcal{S}_2$? 
\end{question}

In addition, Matkowski asked also whether the net $(\mathcal{S}_a: a>0)$ is increasing, cf. \cite[p. 56]{MR3083197}: 
\begin{question}\label{q:matk2}
    Is it true that $\mathcal{S}_a\subseteq \mathcal{S}_b$ for all reals $0<a< b$?
\end{question}

Our main result answers in the negative both Question \ref{q:matk1} and Question \ref{q:matk2}: 
\begin{thm}\label{thm:mainmatk1}
$\mathcal{S}_2\setminus \mathcal{S}_3\neq \emptyset$. In particular, $\mathcal{S}_1\neq \mathcal{S}_2$ and $(\mathcal{S}_a: a>0)$ is not increasing.
\end{thm}

The proof of Theorem \ref{thm:mainmatk1} will be given in Section \ref{sec:proof1}. 
%\textcolor{red}{Add remark non classe C2? See pales1 file..}
As it will be clear from the proof, 
%of Theorem \ref{thm:mainmatk1}, 
the constructed function $f \in \mathcal{S}_2\setminus \mathcal{S}_3$ will be, in addition, even and differentiable at every nonzero $x\in \mathbb{R}$. 
%\textcolor{red}{Add remark non classe C2? See pales1 file..}
The idea of the construction is the following: we will pick a function $f=g+\alpha\,(h-h(0))$, where $g$ is a \textquotedblleft safely\textquotedblright\, subadditive function (hence, $2$-subadditive), and the perturbation $\alpha\,(h-h(0))$ is small enough not to destroy $2$-subadditivity but large enough to violate $3$-subadditivity at a well-chosen pair $(x,y)$. The map $h$ will be a even ``Gaussian ring'' peaked at $|x|=\mu$ with width $\sigma$ of the type
$$
h(x)=e^{-\left(\frac{|x|-\mu}{\sigma}\right)^2},
$$
for some $\mu,\sigma>0$. 
Subtracting $h(0)$ merely normalizes so that the perturbations vanishes at the origin and does not alter gaps from $g$ except by constants. 

\medskip

To state our second result, we recall that a subadditive bijection $f: (0,\infty)\to (0,\infty)$ with $\lim_{x\to 0}f(x)=0$ has to be an homeomorphism of $(0,\infty)$, see \cite[Corollary 2]{MR1176072} and cf. \cite{MR1088646}. On the other hand, among the discontinuous examples in this setting, Matkowski and {\'S}wi{\k a}tkowski proved in \cite[Theorem 2]{MR1176072} that there exists a subadditive bijection $f: (0,\infty)\to (0,\infty)$ such that 
$$
%\forall c>0, \quad \inf f((0,c))>0 \text{ and }\sup f((0,c))<\infty, 
\liminf_{x\to 0}f(x)>0
\quad \text{ and }\quad 
\limsup_{x\to 0}f(x)<\infty, 
$$
cf. also \cite[p. 194]{MR1176072}. %MR1088646}. 
They also remark that \cite[Example 1]{MR1088646} shows the existence of a subadditive bijection $f: (0,\infty)\to (0,\infty)$ such that 
$$
\liminf_{x\to 0}f(x)=0 \quad \text{and }\quad \limsup_{x\to 0}f(x)=\infty. 
$$
Accordingly, they state the following open question: 
\begin{question}\label{q:matk2}
Does there exist a subadditive bijection $f: (0,\infty)\to (0,\infty)$ such that 
$$
\liminf_{x\to 0}f(x)=0 \quad \text{and }\quad 0<\limsup_{x\to 0}f(x)<\infty\,\,?
$$
\end{question}
In \cite[Theorem 3]{MR1176072}, they observed that if $f: (0,\infty)\to (0,\infty)$ is a subadditive bijection and $f^{-1}$ is bounded in a neighborhood of $0$ then $\lim_{x\to 0}f^{-1}(0)=0$. In light of this result, they concluded that Question \ref{q:matk2} ``seems to be rather difficult to decide.'' 

Although we do \emph{not} have a final answer to Question \ref{q:matk2}, we can show the answer is affirmative once we replace $(0,\infty)$ with the positive rational cone $H\subseteq (0,\infty)$ (that is, a subset which is stable under finite sums and multiplications by positive rationals) with an infinite set of generators which are linearly independent over $\mathbb{Q}$. 
%In our second main result, we answer Question \ref{q:matk2} in the affirmative: 
\begin{thm}\label{thm:mainmatk2}
Let $H$ be a positive rational cone generated by a $\mathbb{Q}$-linearly independent infinite subset of $(0,\infty)$. 
%Hamel basis $B\subseteq (0,\infty)$ be a Hamel basis of $\mathbb{R}$ over $\mathbb{Q}$. 
%Let $H\subseteq (0,\infty)$ be a positive rational cone which is not finitely generated. 
Then there exists a subadditive bijection $f: H\to H$ such that 
\begin{equation}\label{eq:claimmatk}
\liminf_{x\to 0}f(x)=0 \quad \text{and }\quad \limsup_{x\to 0}f(x)=1. 
\end{equation}
\end{thm}
The proof of Theorem \ref{thm:mainmatk2} will be given in Section \ref{sec:proof2}.

\section{Proof of Theorem \ref{thm:mainmatk1}}\label{sec:proof1}

As anticipated in Section \ref{sec:intr}, let $f: \mathbb{R}\to \mathbb{R}$ be the function (which depends on the parameters $\mu, \sigma, \alpha \in (0,\infty)$ that will be chosen later) defined by 
\begin{equation}\label{eq:definitionf}
\forall x \in \mathbb{R}, \quad 
f(x):=g(x)+\alpha (h(x)-h(0)),
\end{equation}
%for all $x \in \mathbb{R}$, 
where 
$$
\forall x \in \mathbb{R}, \quad 
g(x):=|x|+\log(1+|x|)
\quad \text{ and }\quad 
h(x):=e^{-\left(\frac{|x|-\mu}{\sigma}\right)^2}. 
$$
%for all $x \in \mathbb{R}$. 
Of course, all $f,g,h$ are continuous and $f(0)=g(0)=(h-h(0))(0)=0$. 

Our proof strategy will be to show that, for a suitable choice of the triple $(\mu, \sigma,\alpha)$, the function $f$ defined in \eqref{eq:definitionf} satisfies the inequality $f(2x+y) \le 2f(x)+f(y)$ for each pair $(x,y)$ in the regions $\mathscr{A}, \mathscr{B}, \mathscr{C}\subseteq \mathbb{R}^2$, where 
%$$
%\mathscr{A}:=\left\{(x,y) \in\mathbb{R}^2: |x| \ge \nicefrac{1}{3}\right\}, \quad 
%\mathscr{B}:=\left\{(x,y) \in\mathbb{R}^2: 2|x|+|y| \le 1\right\},\,\,\,\text{ and }\,\,
%$$
%$$
%\mathscr{C}:=\left\{(x,y) \in\mathbb{R}^2: |x| \le \nicefrac{1}{3} \text{ and }|y| \ge \nicefrac{1}{3}\right\}.
%$$
$$
\mathscr{A}:=\left\{(x,y) \in\mathbb{R}^2: |x| \ge \nicefrac{1}{2}\right\}, \quad 
\mathscr{B}:=\left\{(x,y) \in\mathbb{R}^2: 2|x|+|y| \le 1\right\},\,\,\,\text{ and }\,\,
$$
$$
\mathscr{C}:=\left\{(x,y) \in\mathbb{R}^2: |x| \le  \nicefrac{1}{2} \text{ and }2|x|+|y| \ge 1\right\}.
$$
It is easy to see that $\mathscr{A}\cup \mathscr{B}\cup \mathscr{C}=\mathbb{R}^2$, hence this will provide in Theorem \ref{prop:mainS2} sufficient conditions on the triples $(\mu,\sigma,\alpha)$ to ensure that $f \in \mathcal{S}_2$. Finally, a numerical counterexample will show that $f \notin \mathcal{S}_3$ (and, in particular, $f$ is not subadditive).

\begin{lem}\label{lem:gsubadditive}
$g$ is subadditive, i.e., $g \in \mathcal{S}_1$. In particular, $g \in \mathcal{S}_2$. 
\end{lem}
\begin{proof}
    Pick $x,y \in \mathbb{R}$ and observe that $1+|x+y|\le 1+|x|+|y| \le (1+|x|)(1+|y|)$. Taking the logs and using the triangular inequality, it follows that  
    $$
    g(x+y) \le |x|+|y|+\log((1+|x|)(1+|y|))=g(x)+g(y).
    $$
    Hence $g \in \mathcal{S}_1$. In particular, $g \in \mathcal{S}_2$ by Lemma \ref{lem:basic}. 
\end{proof}

In our main proofs below, we will need also the functions $\phi,\lambda: [0,\infty)\to  \mathbb{R}$ defined by 
$$
\phi(z):=(4z^2-2)e^{-z^2}
\quad \text{ and }\quad 
\lambda(z):=2\log(1+z)-\log(1+2z)
$$
%\,\,\,
for all $z\ge 0$, together with $\psi: [0,1)\to \mathbb{R}$ given by 
$$
\psi(z):=\log((1+z)^2(1-z))
$$
for all $z\in [0,1)$. Also, for each $w: \mathbb{R}\to \mathbb{R}$, let $\Delta_w$ be the function $\mathbb{R}^2\to \mathbb{R}$ defined by 
$$
%\forall x,y \in \mathbb{R}, \quad 
\Delta_w(x,y):=2w(x)+w(y)-w(2x+y) 
$$
for all $(x,y) \in \mathbb{R}^2$. Informally, $\Delta_w$ quantifies the distance of $w$ to $\mathcal{S}_2$. In fact, $w \in \mathcal{S}_2$ if and only if $\Delta_w \ge 0$. 
Observe also, by the above definitions that 
\begin{equation}\label{eq:differencef}
\Delta_f
=\Delta_g+\alpha \Delta_{h-h(0)}=\Delta_g+\alpha(\Delta_h-2h(0)).
\end{equation}
%Hence, we will search o search for parameters $\alpha,\sigma,\mu$ which allow $\Delta_f$ to be nonnegative (so that $f \in \mathcal{S}_2$) and, at the same time, the subadditivity property fails.} 

With these premises, Lemma \ref{lem:gsubadditive} shows that $\Delta_g \ge 0$ in different regions of $\mathbb{R}^2$. In the next lemma, we improve this lower bound. Here, we will write $C$ for the constant 
%and $D$ for the constants \textcolor{red}{SERVE LA C?}
$$
C:=\lambda(\nicefrac{1}{2})=\log\left(\nicefrac{9}{8}\right)\approx 0.11778
%\quad \text{ and }\quad 
%D:=\lambda(\nicefrac{1}{3})=\log\left(\nicefrac{16}{15}\right)\approx 0.06454 
$$
\begin{lem}\label{lem:Deltagimproved} 
%Define the  function $q: [0,\infty) \to \mathbb{R}$ by 
%$$
%\forall z\ge 0, \quad q(z)
%$$
%Fix $\kappa \in (0,1)$. Then 
The following hold\textup{:}
\begin{enumerate}[itemsep=0.1cm, label={\rm (\roman*)}]
\item \label{item:1bounddeltag} $\lambda$ is  increasing\textup{;} 

\item \label{item:2bounddeltag} $\Delta_g(x,y) \ge \lambda(|x|)$ for all $(x,y) \in \mathbb{R}^2$\textup{;}

\item \label{item:3bounddeltag} $\Delta_g(x,y) \ge C$ for all 
$(x,y) \in \mathscr{A}$\textup{;}
%$(x,y) \in \mathbb{R}^2$ with $|x|\ge \nicefrac{1}{2}$\textup{;}

\item \label{item:4bounddeltag} $\Delta_g(x,y) \ge \frac{3}{8}x^2$ for all 
$(x,y) \in \mathscr{B}\cup \mathscr{C}$\textup{;}
%$(x,y) \in \mathbb{R}^2$ with $|x|\le \nicefrac{1}{2}$\textup{;}

\item \label{item:5bounddeltag} $\Delta_g(x,y) \ge \psi(|x|) \ge 2|x|$ for all 
$(x,y) \in \mathscr{C}$\textup{.}
%$(x,y) \in \mathbb{R}^2$ with $|x|\le  \nicefrac{1}{2}$ and $2|x|+|y|\ge 1$\textup{.}
\end{enumerate}
\end{lem}
\begin{proof}
   \ref{item:1bounddeltag}. This follows by the fact that the derivative of $(1+z)^2/(1+2z)$ is $2z(1+z)/(1+2z)^2$, which is nonnegative on $[0,\infty)$. 

   \medskip

   \ref{item:2bounddeltag}. Proceeding as in the proof of Lemma \ref{lem:gsubadditive}, we have \begin{equation}\label{eq:easylowerboundDeltag}
    \begin{split} 
    \Delta_g(x,y)&\ge 2\log(1+|x|)+\log(1+|y|)-\log(1+|2x+y|)\\ 
    &\ge 2\log(1+|x|)-\log(1+2|x|)=\lambda(|x|). \end{split} 
    \end{equation}

    \medskip

    \ref{item:3bounddeltag}. It follows by items \ref{item:1bounddeltag} and \ref{item:2bounddeltag}.

    \medskip

    \ref{item:4bounddeltag}. Pick $(x,y) \in \mathbb{R}^2$ with $|x| \le \nicefrac{1}{2}$. Taking into account that $\log(1+z) \ge z-z^2/2$ for all $z\ge 0$, we obtain by item \ref{item:2bounddeltag} that 
    \begin{displaymath}%\label{eq:estimateiuhlfuhg} 
   \begin{split} 
   \Delta_g(x,y)&\ge\lambda(|x|)=\log((1+|x|)^2)-\log(1+2|x|)\\ 
   &=\log\left(1+\frac{x^2}{1+2|x|}\right) 
   %\\& 
   \ge \frac{x^2}{1+2|x|}-\frac{1}{2}\left(\frac{x^2}{1+2|x|}\right)^2\\ 
   &\ge \frac{x^2}{2}-\frac{x^4}{2} \ge \frac{x^2}{2}-\frac{x^2}{8}=\frac{3}{8}\,x^2. \end{split} 
   \end{displaymath} 

   \medskip

   \ref{item:5bounddeltag}. Using the triangular inequality at the intermediate function in \eqref{eq:easylowerboundDeltag} and recalling that $2|x|+|y|\ge 1$, we have 
   \begin{displaymath}
       \begin{split}
           \Delta_g(x,y) 
           &\ge 2\log(1+|x|)+\log(1+|y|)-\log(1+2|x|+|y|)\\
           &\ge \inf_{|z| \ge 1-2|x|}\left(2\log(1+|x|)+\log(1+|z|)-\log(1+2|x|+|z|)\right)\\
           &\ge 2\log(1+|x|)+\log(2-2|x|)-\log(2)=\psi(|x|).
       \end{split}
   \end{displaymath} 
   (In the estimate with the infimum we used the function was increasing in $z$.) Thus, since $\log(1+z) \ge z-z^2/2$ for all $z\ge 0$ and $\log(1+z)\le z$ for all $z>-1$, we conclude that 
   $$
   %\forall z \in [0,1), \quad 
   \psi(z)=2\log(1+z)-\log(1-z)\ge 2\left(z-\frac{z^2}{2}\right)+z \ge 3z-z^2 \ge 2z
   $$
   for all $z \in [0,\nicefrac{1}{2}]$. 
\end{proof}

\subsection{Region $\mathscr{A}$} First, we will obtain sufficient conditions on $(\mu,\sigma,\alpha)$ to ensure that $\Delta_f\ge 0$ on the region $\mathscr{A}$. 
\begin{lem}\label{lem:Deltah}
  $-1\le \Delta_h(x,y) \le 3$ for all $(x,y) \in \mathbb{R}^2$.  
    %$|\Delta_h(x,y)| \le 4e^{-C}$ for all $(x,y) \in \mathbb{R}^2$, where $C:=(\mu/\sigma)^2$. 
\end{lem}
\begin{proof}
    It is enough to recall the definition of   
    $
    \Delta_h%(x,y)= 2h(x)+h(y)-h(2x+y)
    $ 
    and that $0\le h\le 1$.
\end{proof}

\begin{prop}\label{prop:largeregime}
    Fix parameters $\alpha,\mu,\sigma \in (0,\infty)$ such that 
    $$
    \alpha \le \frac{C}{1+2e^{-(\mu/\sigma)^2}}. %,
    $$
    Then 
    $\Delta_f(x,y) \ge 0$ for all $(x,y) \in \mathscr{A}$. 
\end{prop}
\begin{proof}
    Taking into account Lemma \ref{lem:Deltagimproved}\ref{item:3bounddeltag} and Lemma \ref{lem:Deltah}, we obtain by \eqref{eq:differencef} that 
    \begin{displaymath}
        \begin{split}
            \Delta_f(x,y) &\ge C+\alpha (-1-2h(0))
            %\\&=C-2\alpha(2+e^{-(\mu/\sigma)^2})
            \ge 0
        \end{split}
    \end{displaymath}
    for all $(x,y) \in \mathbb{R}^2$ with $|x|\ge \nicefrac{1}{2}$. 
\end{proof}

\subsection{Region $\mathscr{B}$}
Before we provide sufficient conditions to ensure that $\Delta_f\ge 0$ on the region $\mathscr{B}$, we recall the following elementary result. Here, we provide a self-contained proof for the sake of completeness, cf. also \cite[Theorem 3.2]{MR1007135} for generalizations in this direction. 
\begin{lem}\label{lem:atkinson2}
    Let $r: \mathbb{R}\to \mathbb{R}$ be a function with continuous second derivative. Then
    %, for all $t>0$, there exists $\xi_t \in (0,t)$ such that 
    $$
    \forall t>0, \exists \xi_t \in (0,t), \quad \quad 
    2r\left(\frac{t}{2}\right)-r(t)=
    r(0)-\frac{t^2}{4}r^{\prime\prime}(\xi_t).
    $$
\end{lem}
\begin{proof}
    Fix $t>0$ and define $q: [0,1]\to \mathbb{R}$ by $q(x):=r(xt)$ for each $x \in [0,1]$. Let also $p(x):=ax^2+bx+c$ be the unique quadratic polynomial such that $q(x)=p(x)$ for all $x \in \{0,1/2,1\}$. Hence, the function $\kappa:=q-p$ annihilates on $\{0,1/2, 1\}$ and it has continuous second derivative. By Rolle's theorem, there exist $\eta_1\in(0,1/2)$ and $\eta_2\in(1/2,1)$ with $\kappa^\prime(\eta_1)=\kappa^\prime(\eta_2)=0$. Applying Rolle again on $[\eta_1,\eta_2]$ yields $\xi\in(\eta_1,\eta_2)\subseteq (0,1)$
such that
\begin{equation}\label{eq:kappaprimeprime}
\kappa^{\prime\prime}(\xi)=0.
\end{equation}
Then, it will be enough to show that $\xi_t:=\xi t\in (0,t)$ satisfies our claim. To this aim, observe that \eqref{eq:kappaprimeprime} implies $0=q^{\prime\prime}(\xi)-p^{\prime\prime}(\xi)$, so that $q^{\prime\prime}(\xi)
%=p^{\prime\prime}(\xi)
=2a$. Hence we obtain
\begin{displaymath}
    \begin{split}
        2q\left(\frac{1}{2}\right)-q(1)-q(0)
        &=2p\left(\frac{1}{2}\right)-p(1)-p(0)\\
        &=2\left(\frac{a}{4}+\frac{b}{2}+c\right)-(a+b+c)-c
        %\\&
        =-\frac{1}{2}a=-\frac{1}{4}q^{\prime\prime}(\xi).
    \end{split}
\end{displaymath}
The claim follows recalling the definition of $q$. 
\end{proof}

\begin{lem}\label{lem:increasingDeltaf}
    Fix parameters $\alpha, \mu, \sigma \in (0,\infty)$ with $\mu\ge 1$. Then $\Delta_f(x,y) \ge \Delta_f(|x|,|y|)$ for all $(x,y) \in \mathscr{B}$. 
    %$(x,y) \in \mathbb{R}^2$ with $2|x|+|y|\le 1$. 
\end{lem}
\begin{proof}
    Pick $(x,y) \in \mathbb{R}^2$ with $2|x|+|y|\le 1$. It follows by \eqref{eq:differencef} that 
    $$
    \Delta_f(x,y)-\Delta_f(|x|,|y|)=-f(2x+y)+f(2|x|+|y|).
    $$
    Considering that $|2x+y| \le 2|x|+|y|$ and that $f$ is even, then it is sufficient to show that $f$ is increasing on $[0,1]$, provided that $\mu\ge 1$. In fact, we have 
    $$
    f^\prime(t)=1+\frac{1}{1+t}+\frac{2\alpha}{\sigma^2}(\mu-t)h(t)\ge 1+\frac{1}{1+t}>0
    $$
    for all $t \in (0,1]$. This concludes the proof.
\end{proof}

%\textcolor{red}{Add definition $\phi$} % which depends on $\mu$ and $\sigma$}
\begin{prop}\label{prop:smallregime}
    Fix parameters $\alpha, \mu, \sigma \in (0,\infty)$ such that 
    $$
    \mu\ge 1+\sigma\sqrt{3/2} 
    \quad \text{ and }\quad 
    \alpha \le \frac{17\sigma^2}{54\phi((\mu-1)/\sigma)},
    $$
    Then $\Delta_f(x,y) \ge 0$ for all $(x,y) \in \mathscr{B}$.  
    %$(x,y) \in \mathbb{R}^2$ with $2|x|+|y|\le 1$. 
\end{prop}
\begin{proof}
    By the definition of $h$ (which depends only on $\mu$ and $\sigma$) we have that 
    $$
    h^{\prime\prime}(x)
    =\frac{1}{\sigma^2}\phi\left(\frac{||x|-\mu|}{\sigma}\right)
    %=\frac{1}{\sigma^2}\left(4\tilde{x}^2-2\right)e^{-\tilde{x}^2}
    %=2\sigma^{-4}\left(2(x-\mu)^2-\sigma^2\right)h(x)
    $$
    on $\mathbb{R}\setminus \{0\}$. %on $(0,\infty)$. 
    %,  where $\tilde{x}:=(x-\mu)/\sigma$. 
    Hence $h$ is convex on 
    %the interval 
    $[0,\mu-\sigma/\sqrt{2}]$. Since, in particular $\mu\ge 1+\sigma/\sqrt{2}$ by the hypothesis, it is convex on $[0,1]$. Pick $(x,y) \in \mathbb{R}^2$ and suppose that $(x,y) \in \mathscr{B}$, namely, 
    $$
    t:=2|x|+|y| \le 1.
    $$ 
    We claim that $\Delta_f(x,y)\ge 0$. Thanks to Lemma \ref{lem:increasingDeltaf}, we have that $\Delta_f(x,y)\ge \Delta(|x|,|y|)$ if $\mu\ge 1$, hence we can assume hereafter without loss of generality that $x,y\ge 0$; 
    %If $t$ is kept fixed, then we can assume without loss of generality that $x,y \ge 0$ (since all functions $f,g,h$ are even); 
    in particular, $x \in [0,t/2]\subseteq [0,1/2]$. 

    At this point, define the function $\tau_t: [0,t/2] \to \mathbb{R}$ by 
    $$
    \forall x \in [0,t/2], \quad 
    \tau_t(x):=\Delta_f(x,t-2x).
    $$
    Considering that $\mu\ge 1+\sigma\sqrt{3/2}$, we obtain by elementary calculus that 
    $$
    M:=\sup_{u \in (0,1]}h^{\prime\prime}(u)=h^{\prime\prime}(1)=\frac{1}{\sigma^2}\phi\left(\frac{\mu-1}{\sigma}\right). 
    $$
    %the global maximum of the function $h^{\prime\prime}$ is attained on $(0,\infty)$ at the point $\tilde{x}=\sqrt{3/2}$, with value $4e^{-3/2}/\sigma^2$. 
    Hence, we get by \eqref{eq:differencef} that 
    \begin{displaymath}
        \begin{split}
\tau_t^{\prime\prime}(x)&=2g^{\prime\prime}(x)+4g^{\prime\prime}(t-2x)+\alpha (2h^{\prime\prime}(x)+4h^{\prime\prime}(t-2x))   \\
&\le \frac{-2}{(1+x)^2}+\frac{-4}{(1+t-2x)^2}+6\alpha M\\
&\le \frac{-2}{(1+t/2)^2}+\frac{-4}{(1+t)^2}+6\alpha M\\
&\le -\frac{8}{9}-1+6\alpha M  \\
&\le -\frac{17}{9}+6M \cdot \left(\frac{17}{54M} \right) =0
        \end{split}
    \end{displaymath}
    for all $x \in (0,t/2]$. It follows that $\tau_t$ is a continuous concave function, hence its minimum has to be at the boundary points of its domain. Of course, $\tau_t(0)=2f(0)+f(t)-f(t)=0$. In addition, thanks to Lemma \ref{lem:Deltagimproved}\ref{item:4bounddeltag} and Lemma \ref{lem:atkinson2}, at the second endpoint $x=t/2$ (so that $y=0$ since $2x+y=t$) we have that there exists $\xi_t \in (0,t)$ for which
    \begin{displaymath}
        \begin{split}
            \tau_t\left(\frac{t}{2}\right)
            &=2f\left(\frac{t}{2}\right)+f(0)-f(t)\\
            &=2g\left(\frac{t}{2}\right)-g(t)+\alpha\left(2h\left(\frac{t}{2}\right)-h(t)-h(0)\right)\\
            &=\lambda\left(\frac{t}{2}\right)+\alpha\left(-\frac{t^2}{4}h^{\prime\prime}(\xi_t)\right)\\
            &\ge \frac{3}{8}\left(\frac{t}{2}\right)^2+\alpha\left(-\frac{t^2}{4}M\right)\\
            &
            %=t^2\left(\frac{3}{32}-\frac{\alpha}{e^{3/2}\sigma^2}\right)
            \ge t^2\left(\frac{3}{32}-\frac{M}{4}\cdot \frac{17}{54 M}\right) \ge \frac{t^2}{100} \ge 0.
        \end{split}
    \end{displaymath}
Therefore $\Delta_f(x,y)\ge 0$. 
\end{proof}

\subsection{Region $\mathscr{C}$} Finally, we provide sufficient conditions to ensure that $\Delta_f\ge 0$ on  $\mathscr{C}$. 
\begin{prop}\label{prop:mixed}
    Fix parameters $\alpha, \mu, \sigma \in (0,\infty)$ such that 
    $$
    \mu\ge \nicefrac{1}{2}
    \quad \text{ and }\quad 
    \alpha \le 
    \sigma \sqrt{e/2}.
    $$
    Then $\Delta_f(x,y) \ge 0$ for all $(x,y) \in \mathscr{C}$. 
\end{prop}
\begin{proof}
    By standard calculations we have that $h^\prime(x)=2h(x)(\mu-x)/\sigma^2$ for all $x>0$ (and $h$ is even). Hence $h(x) \ge h(0)$ for all $|x|\le \nicefrac{1}{2}$. 
    Now fix $(x,y) \in \mathscr{C}$, so that 
    $|x|\le \nicefrac{1}{2}$ and $2|x|+|y|\ge 1$. 
    %Fix also $(x,y) \in \mathbb{R}^2$ with $|x|\le \nicefrac{1}{2}$ and $2|x|+|y|\ge 1$. 
    Notice that $|h(y)-h(2x+y)|\le |y-(2x+y)| \sup |h^\prime|$. 
    %and $|h(x)-h(0)|\le |x|\sup |h^\prime|$. 
    Putting it together with Lemma \ref{lem:Deltagimproved}\ref{item:5bounddeltag}, we obtain that 
    \begin{displaymath}
        \begin{split}
            \Delta_f(x,y)&=\Delta_g(x,y)+\alpha \Delta_{h-h(0)}(x,y)\\
            &=\Delta_g(x,y)+\alpha(h(y)-h(2x+y))+2\alpha(h(x)-h(0))\\
            &\ge 2|x|-2\alpha|x| \sup |h^\prime|%-2\alpha|x| \sup |h^\prime|
            \\
            &=2|x|\left(1-\alpha \frac{\sqrt{2/e}}{\sigma} \right) \ge 0.
            \end{split}
    \end{displaymath}
    This concludes the proof. 
\end{proof}

\subsection{Conclusion} Merging together the above results, we obtain: 
\begin{thm}\label{prop:mainS2}
    Fix parameters $\alpha, \mu, \sigma \in (0,\infty)$ such that 
    $$
    \mu\ge 1+\sigma\sqrt{3/2} 
    %\frac{\sigma}{\sqrt{2}}
    \quad \text{ and }\quad 
    \alpha \le \min\left\{
    \frac{17\sigma^2}{54\phi((\mu-1)/\sigma)}, 
    %\frac{e^{3/2}\sigma^2}{13}, 
    \frac{C}{1+2e^{-(\mu/\sigma)^2}}, \sigma \sqrt{e/2}\right\}.
    $$
    Then $f \in \mathcal{S}_2$. 
\end{thm}
\begin{proof}
    It follows putting together Proposition \ref{prop:largeregime}, Proposition \ref{prop:smallregime}, and Proposition \ref{prop:mixed}. 
\end{proof}

This allows to complete the proof of Theorem \ref{thm:mainmatk1}.
\begin{proof}
Let $f$ be function defined in \eqref{eq:definitionf} corresponding to the values 
$$
\mu=1.2, \quad \sigma=0.05, \quad \text{ and }\quad \alpha=0.05. 
%(\mu, \sigma,\alpha)=(1.2, 0.05, 0.04).
$$
Then we obtain that: 
\begin{enumerate}[itemsep=0.1cm, label={\rm (\roman*)}]
\item $\mu=1+4\sigma> 1+\sigma\sqrt{3/2}$; 
\item Since $(\mu-1)/\sigma=4$, we get
$$
\frac{17\sigma^2}{54\phi((\mu-1)/\sigma)}
=\frac{17}{54\cdot 20^2 \cdot \phi(4)}
=\frac{17\cdot e^{16}}{54\cdot 20^2 \cdot (4^3-2)}>\frac{2^4\cdot 2^{16}}{2^6\cdot 2^9 \cdot 2^6}=\frac{1}{2}>\alpha; 
$$
\item Since $\mu/\sigma=24$, we get 
$$
\frac{C}{1+2e^{-(\mu/\sigma)^2}}>\frac{\nicefrac{1}{10}}{1+2e^{-24^2}}>\frac{1}{9}>\alpha; 
$$
\item $\sigma\sqrt{e/2}>\sigma=\alpha$. 
\end{enumerate}
It follows by Theorem \ref{prop:mainS2} that $f \in \mathcal{S}_2$. 

Lastly, suppose that 
$x_\star=0.016$ and $y_\star=1.137$. Then 
$$
f(3x_\star+y_\star)-3f(x_\star)-f(y_\star)> 0.01 >0. 
$$
Therefore $f\notin \mathcal{S}_3$, concluding the proof. 
\end{proof}

\begin{rmk}
    Additional numerical examples of triples $(\mu,\sigma,\alpha)$, with a given value of $\alpha$, for which the function $f$ defined in \eqref{eq:definitionf} belongs to $\mathcal{S}_2\setminus \mathcal{S}_3$ can be found in the table below. 

    \medskip
    
\begin{center}
\begin{tabular}{@{\extracolsep{8pt}} cccccc}
\hline
  $\mu$ & $\sigma$ & $\alpha$ & $x_\star$ & $y_\star$ & $f(3x_\star+y_\star)-3f(x_\star)-f(y_\star)$ \\
\hline\hline %\\
  \rule{0pt}{3ex} 
  %1.2 & 0.05 & 0.117783036 & 0.01010 & 1.14817 & 0.006233979 \\
  %1.3 & 0.05 & 0.117783036 & 0.00860 & 1.24777 & 0.005452333 \\
  1.5 & 0.05 & 0.117783036 & 0.00675 & 1.45367 & 0.001664770 \\
  2.0 & 0.10 & 0.117783036 & 0.01050 & 1.95491 & 0.000326430 \\
  2.5 & 0.10 & 0.117783036 & 0.00900 & 2.45647 & 0.000183238 \\
  3.0 & 0.10 & 0.117783036 & 0.00750 & 2.95886 & 0.000105165 \\
  %2.0 & 0.15 & 0.117783036 & 0.01500 & 1.95769 & 0.000162218 \\
  %2.5 & 0.15 & 0.117783036 & 0.01275 & 2.45932 & 0.000130739 \\
  %3.0 & 0.15 & 0.117783036 & 0.01125 & 2.96203 & 0.000082451 \\
  5.0 & 0.15 & 0.117783036 & 0.00750 & 4.96456 & 0.000053255 \\
  %1.8 & 0.12 & 0.117783036 & 0.01200 & 1.75743 & 0.000176707 \\
  %2.0 & 0.20 & 0.117783036 & 0.02000 & 1.96102 & 0.000126411 \\
  %2.5 & 0.20 & 0.117783036 & 0.01700 & 2.46262 & 0.000099961 \\
  %3.0 & 0.20 & 0.117783036 & 0.01500 & 2.96564 & 0.000077907 \\
  %5.0 & 0.20 & 0.117783036 & 0.01000 & 4.96837 & 0.000049730 \\
  \hline
\end{tabular}
\end{center}
\end{rmk}

\section{Proof of Theorem \ref{thm:mainmatk2}}\label{sec:proof2}

%We will obtain Theorem \ref{thm:mainmatk2} as an immediate consequence of Theorem \ref{thm:mainmatk3} below. 
The proof of Theorem \ref{thm:mainmatk2} will essentially rely on a Hamel basis construction of the vector space $\mathbb{R}$ over $\mathbb{Q}$ and piecewise linear (concave) bijections on countably many rational rays, glued with the identity elsewhere. 

Hereafter, given a subset $S\subseteq \mathbb{R}$, we write $\mathrm{span}_{\mathbb{Q}}(S)$ and $\mathrm{span}_{\mathbb{Q}_+}(S)$ for the rational span and the positive rational span, respectively. Also, $\mathbb{N}_+:=\{1,2,\ldots\}$. 
%\begin{thm}\label{thm:mainmatk3}
%Let $H\subseteq \mathbb{R}$ be an infinite dimensional $\mathbb{Q}$-vector subspace. Then there exists a subadditive bijection $f: H_+\to H_+$ which satisfies \eqref{eq:claimmatk}. 
%$\liminf f(x)=0$ and $\limsup f(x)=1$ as $x\to 0$.
%\end{thm}
\begin{proof}
[Proof of Theorem \ref{thm:mainmatk2}]
By hypothesis, it is possible fix a $\mathbb{Q}$-linearly independent infinite set $B\subseteq (0,\infty)$ whose positive rational cone is $H$. 
Pick a countably infinite subset $B_0:=\{p_n: n \in \mathbb{N}_+\}\subseteq B$. % such that $B\setminus B_0\neq \emptyset$. %, and define
%$$
%U:=\mathrm{span}_{\mathbb{Q}_+}(B_0) 
%\quad \text{ and }\quad 
%V:=\mathrm{span}_{\mathbb{Q}_+}(B\setminus B_0). 
%$$
%It is clear that every $x \in H$ admits a unique decomposition $u+v$, with $u \in U$ and $v \in V$. Thus, we write $\pi_U$ and $\pi_V$ for the  projections $\pi_U(x):=u$ and $\pi_V(x):=v$.
%
Multiplying by suitable rationals, if necessary, we can assume without loss of generality that $0<p_n<2^{-n}$ for all $n \in \mathbb{N}_+$. For each $n \in \mathbb{N}_+$,  
define $P_n:=\mathrm{span}_{\mathbb{Q}_+}(\{p_n\})$, $P:=\bigcup_n P_n$, 
and pick $q_n \in \mathbb{Q}$ such that $1-2^{-n} < p_nq_n<1$. In particular, we have $q_n>(1-2^{-n})/p_n>2^n-1\ge 1$ for each $n\in \mathbb{N}_+$. Now, define the map $f_n: P_n\to P_n$ by 
\begin{displaymath}
\forall x \in P_n, \quad 
    f_n(x):=
    \begin{cases}
       \,q_nx\,\,\,\,& \text{ if }\, x\le p_n\\
        \,x+p_n(q_n-1) & \text{ otherwise. }
    \end{cases}
\end{displaymath}
\begin{claim}\label{claim:fnsubadditive}
    $f_n$ is a subadditive bijection on $P_n$. In addition, $f_n(x)\ge x$ for all $x \in P_n$. 
\end{claim}
\begin{proof}
    It is immediate to see that $f_n$ is a continuous concave piecewise linear map, whose graph is the restriction on $P_n\times \mathbb{R}$ of the segment connecting $(0,0)$ and $(p_n,p_nq_n)$ and the line passing through the latter point and parallel to the main diagonal. Since $q_n \ge 1$, we have $f_n(x) \ge x$ for all $x \in P_n$. It is routine to see that $f_n$ is a bijection. Lastly, it is well known that concave nonnegative functions are subadditive, see e.g. \cite[Theorem 7.2.5]{MR423094}. 
\end{proof}

%\bigskip

At this point, define the map $f: H\to H$ by 
\begin{displaymath}
\forall x \in H \quad 
    f(x):=
    \begin{cases}
       \,f_n(x)\,\,\,& \text{ if }\, x\in P_n \text{ for some }n \in \mathbb{N}_+\\
        \,x & \text{ otherwise. }
    \end{cases}
\end{displaymath}

%Now, define $\hat{f}: U\to U$ by acting coordinatewise on the finite $B_0$-representation: if $x\in U$ can be written (uniquely) as $x=\sum_n r_np_n$, with all $r_n$ positive rationals and all but finitely many are zero, then 
%$$
%%%%\forall u \in U, \quad 
%\hat{f}(x):=\sum_{n\in \mathbb{N}} f_n(r_np_n). 
%$$
%Since each $f_n$ modifies only the $B_0$-coordinates on that line, $\hat{f}$ is a bijection on $U$ with inverse obtained by inverting the $f_n$ coordinatewise. 

%Finally, we define the function $f: H\to H$ by 
%$$
%%%%%\forall x \in H, \quad 
%f(x):=\hat{f}(\pi_U(x))+\pi_V(x). 
%f:=\hat{f}\circ \pi_U + \pi_V.
%$$

We are left to show that $f$ satisfies the required properties.
\begin{claim}\label{claim:fsubadditive}
    $f$ is a subadditive bijection on $H$. 
\end{claim}
\begin{proof}
Each $f_n$ is a bijection on $P_n$ thanks to Claim \ref{claim:fnsubadditive}, hence the restriction of $f$ on $P$ is a bijection. Since $f$ is the identity on $H\setminus P$, it follows that $f$ is bijection on $H$. 

    To show that $f$ is subadditive, fix $x,y \in H$. Observe that, if $y\notin P$, then its decomposition $\sum_{p \in B}r_pp$ satisfies $r_p>0$ for some $p\notin B_0$ or $r_p,r_{p^\prime}>0$ for some disinct $p,p^\prime \in B_0$. Hence, in both instances, we have $x+y\notin P$. Then we have the following cases: 
    \begin{enumerate}[itemsep=0.1cm, label={\rm (\roman*)}]
\item Suppose that $\{x,y\}\subseteq P_n$ for some $n \in \mathbb{N}_+$. Then $x+y \in P_n$, hence $f(x+y)=f_n(x+y) \le f_n(x)+f_n(y)=f(x)+f(y)$ by Claim \ref{claim:fnsubadditive}. 
\item Suppose that $x \in P_n$ and $y \in P_m$ for some distinct $n,m \in \mathbb{N}_+$. Since $\{p_n: n \in \mathbb{N}_+\}$ is linearly independent, it follows that $x+y \notin P$. It follows by Claim \ref{claim:fnsubadditive} that $f(x+y)=x+y \le f_n(x)+f_m(y)=f(x)+f(y)$. 
\item Suppose that $x \in P_n$ and $y\notin P$ for some $n \in \mathbb{N}_+$ (or viceversa). Since $x+y\notin P$, it follows by Claim \ref{claim:fnsubadditive} that $f(x+y)=x+y \le f_n(x)+f(y)=f(x)+f(y)$. 
\item Suppose that $x,y \notin P$. Then $x+y\notin P$ and $f(x+y)=x+y=f(x)+f(y)$. 
\end{enumerate}
Therefore, in all cases, we obtain $f(x+y) \le f(x)+f(y)$. 
\end{proof}

To conclude, we need to prove that $f$ satisfies \eqref{eq:claimmatk}. To this aim, fix $x\notin P$. Taking into account that $f$ is nonnegative that $\lim_n f(p/n)=\lim_np/n=0$, it follows that 
$$
\liminf_{x\to 0}f(x)=0. 
$$
At the same time, by construction the sequence $(p_n)$ satisfies $\lim_np_n=0$ and $\lim_n f(p_n)=\lim_n q_np_n=1$, hence $\limsup_{x\to 0}f(x)\ge 1$. On the other hand, we have also $\limsup_{x\to 0}f(x)\le 1$: in fact, pick $\varepsilon \in (0,1)$ and an arbitrary $x \in (0,\varepsilon)$. If $x \notin P$ then $f(x)=x<\varepsilon<1+\varepsilon$. If $x \in P_n$ for some $n \in \mathbb{N}_+$, then $f(x)=f_n(x) \le x+p_n(q_n-1)<x+1-p_n<1+\varepsilon$. Putting everything together, we obtain that 
$$
\limsup_{x\to 0}f(x)=1, 
$$
which completes the proof. 
\end{proof}

%%%%%%%%%%%%%%%%%%%%%%%%%%%%%%
%\nocite{*}
\bibliographystyle{amsplain}
\bibliography{idealee}
%\begin{thebibliography}{99}
%\end{thebibliography}

\end{document}